\documentclass[11pt,a4paper]{amsart}

\usepackage{amscd,amssymb,amsthm}
\usepackage{graphicx}

\numberwithin{equation}{section}
\usepackage{amssymb}
\usepackage{latexsym}
\usepackage{amsmath}
\usepackage{euscript}
\usepackage{graphics}
\usepackage[all]{xy}

\newcommand\al\alpha
\newcommand\be\beta
\newcommand\ga\gamma
\newcommand\de\delta
\newcommand\tha\theta
\newcommand\la\lambda
\newcommand\si\sigma
\newcommand\om\omega

\newcommand\iy\infty

\numberwithin{equation}{section}

\theoremstyle{plain}
\newtheorem{theorem}{Theorem}[section]

\newtheorem{lemma}[theorem]{Lemma}
\newtheorem{proposition}[theorem]{Proposition}

\theoremstyle{definition}
\newtheorem{definition}{Definition}[section]

\theoremstyle{remark}
\newtheorem{Remark}{Remark}

\numberwithin{equation}{section}

\numberwithin{table}{section}

\numberwithin{figure}{section}

\newcommand{\bea}{\begin{eqnarray*}}
\newcommand{\eea}{\end{eqnarray*}}
\newcommand{\bean}{\begin{eqnarray}}
\newcommand{\eean}{\end{eqnarray}}

\newcommand\dstyle\displaystyle

\newcommand\bma{\begin{pmatrix}}
\newcommand\ema{\end{pmatrix}}
\begin{document}

\title{Duhamel convolution product in the setting of Quantum calculus}
\author{F. Bouzeffour and M. T. Garayev}
\address{Department of mathematics, College of Sciences, King Saud University,
 P. O Box 2455 Riyadh 11451, Saudi Arabia.}

 \email{fbouzaffour@ksu.edu.sa; mgarayev@ksu.edu.sa}
 \thanks{ {\em  2010 Mathematics Subject Classification.}.Primary 33D45, secondary 96J15.
 \hfil\break\indent
 {\em  Key words and phrases.} Duhamel product, $q$-difference operator, $q$-integral $q$-special functions, $q$-Duhamel product.}
\maketitle

\begin{abstract}
In this paper we introduce the notions of $q$-Duhamel product and $q$-integration operator. We prove that the classical Wiener algebra $W(\mathbb{D})$ of all analytic functions on the unit disc $\mathbb{D}$ of the complex plane $\mathbb{C}$ with absolutely convergent Taylor series is a Banach algebra with respect to $q$-Duhamel product. We also describe the cyclic vectors of the $q$-integration operator on $W(\mathbb{D})$ and characterize its commutant in terms of the $q$-Duhamel product operators.
\end{abstract}

\section{Introduction}
From the seventies, the interest on the $q$-deformation theory and the so-called quantum calculus have witnessed a great development, due to their role in many areas such as physics and quantum groups. Taking account of the work of Jackson \cite{Jac1,Jac2} and many authors such as Askey, Gasper \cite{GR}, Ismail \cite{ism3}, Koornwinder \cite{Koornwinder} have recently developed this topic.\\For
instance $q$-convolution structure  and $q$-operational calculus in some functional spaces are one of this interest. However, in literature few papers studied these subject \cite{Al, isma2, bou}.\\
\indent The present article is devoted to the study of the $q$-analogue of
the Duhamel convolution which see Wigley \cite{Wigley, Wig1} is defined as the derivative of the classical Mikusinski convolution product:\begin{equation}f\star
g(x)=\frac{d}{dx}\int_0^xf(x-t)g(t)dt.
\end{equation}
where $f,\,g$ are functions in suitable classes of functions on the segment of the real axis.\\ This convolution plays an important role in operator calculus of Mikusinski \cite{Mins}. Dimovski \cite{Dimo} and Bojinov \cite{Bojin} had good achievements in applications of Duhamel convolution product in many questions of analysis including the theory of multipliers of some classical algebras of functions. In the last decay, the Duhamel product has been extensively explored on various spaces of functions, including $L^p(0,1)$, $C^\infty(0,1)$ $W^{(n)}(0,1)$, $W(\mathbb{D})$ by Karaev and his collaborators see \cite{gara}. In this work we will introduce the $q$-Duhamel product and $q$-integration operator, and we prove that the Wiener algebra $W(\mathbb{D})$ of analytic functions is also a Banach algebra under this new $q$-Duhamel product. We also study the cyclic vectors and commutant of $q$-integration operator.
\section{Preliminaries}
We assume that $z\in{\bf C}$ and $0<q<1$, unless otherwise is
specified. We recall some notations \cite{GR}. For an arbitrary
complex number $a$
$$
(a,q)_n :=\left\{
\begin{array}{lcl}
1&{\rm for}&n=0\\
(1-a)(1-aq)\ldots(1-aq^{n-1})&{\rm for}&n\ge1,\\
\end{array}
\right.
$$
$$
(a,q)_\infty :=\lim_{n\to\infty}(a,q)_n,
$$
and
$$
\left[\begin{array}{c} n \\k
\end{array}
\right ]_{q} :=\frac {[n]_q!} {[k]_q![n-k]_q!},
$$
where $$[n]_q=\frac{(q;q)_n}{(1-q)^n}.$$
Let $W(\mathbb{D})$
denote Wiener disc-algebra of all functions $f(z)=\sum_{n=0}^\infty a_nz^n$, satisfying \begin{equation}
\|f\|_w=\sum_{n=0}^\infty
|a_n|<\infty.\end{equation}
It is well known that $W(\mathbb{D})$ is a Banach algebra with
respect to the pointwise multiplication of functions (i.e., with respect to the usual
Cauchy product (convolution product) of formally power series).\\ \indent The Jackson $q$-integral of a function $f(z)\in W(\mathbb{D})$ on the interval $[0,\xi]$ ($\xi\in \mathbb{D}$) is defined as follows \cite{Jac2}:
\begin{align}
\int_0^z f(x)\,d_qx:=z(1-q)\sum_{n=0}^\infty f(zq^n)q^n.\label{integ}
\end{align}
Also we need the $q$-integration by parts formula:\\If $f$, $g\in W(\mathbb{D})$ are $$ \int_{0}^z(D_qf)(t)\,g(t)\,d_qt=
f(z)\,g(q^{-1}z)-f(0)\,g(0)
-\int_{0}^z f(t)\,(D_q^+g)(t)\,d_qt,
$$
where the {\sl backward\/} and {\sl forward $q$-derivatives} are defined by
$$
(D_qf)(z):={f(z)-f(qz)\over(1-q)z},\quad
(D_q^+f)(z):={f(q^{-1}z)-f(z)\over(1-q)z}.$$
Consider the $q$-exponentials, see \cite{GR} defined by
\begin{equation*}\label{2.1}
e_{q}(z) :=\sum_{n=0}^\infty\frac{z^n}{(q,q)_n}=
\frac1{(z,q)_\infty},\qquad|z|<1
\end{equation*}
and
\begin{equation}
E_{q}(z) :=\sum_{n=0}^\infty\frac{q^{n\choose2}z^n}{(q,q)_n}=
(-z,q)_\infty.\label{2.2}
\end{equation}
The $q$-exponentials functions $e_{q}(z)$ and $E_{q}(z)$ satisfy
$$ D_qe_{q}((1-q)z)=e_{q}((1-q)z)\,\,\,\mbox{and}\,\,\,D_{q}E_{q}((1-q)z)=E_{q}(q(1-q)z).$$

\section{Wiener Banach algebra}
In this section, we define the $q$-translation operator and  $q$-Duhamel product related to $q$-difference operator $D_q,$ and we show that $W(\mathbb{D})$ is a Banach algebra with multiplication as $q$-Duhamel product.
\subsection{$q$-Translation}
Let $\xi \in \mathbb{C}$, the $q$-translation operator  $\tau_q^\xi$ is defined on monomials $z^n$ by \cite{ism3}
\begin{equation}
\tau_q^\xi1=1,\quad\tau_q^\xi z^n:=(z+q\xi)\dots(z+q^n\xi),\quad n=1,\,2,\dots\,\,.
\end{equation}
It is clearl that
\begin{align}
\tau_q^{q^r\xi} (q^sz)^n=q^k\tau_q^{ q^{s-r}\xi} z^n=q^s\tau_q^{\xi} (q^{r-s}z)^n
\end{align}
and \begin{equation}
\tau_q^{-z/q^{r}} z^n=0=\tau_q^{\xi} (-q^rz)^n,\quad r=1,\,\dots, n.
\end{equation}
Note that \begin{align}
\lim_{q\rightarrow 1}\tau_q^\xi z^n&=\lim_{q\rightarrow 1}\prod_{k=0}^{n-1}(z+\xi q^{k+1})=(z+\xi)^n.
\end{align}

\begin{proposition} Let $\xi \in \mathbb{C} $. The $q$-translation operator $\tau_q^\xi $ can be extended to the function $f(z)=\sum_{n=0}^{\infty}a_nz^n\in W(\mathbb{D})$ as follows
\begin{equation}
\tau_q^\xi f(z)=\sum_{n=0}^{\infty}a_n (z+q\xi)\dots(z+q^n\xi)
.\label{translation2}
\end{equation}
Moreover, the function $\tau_q^\xi f(z)$ as a function of $z$ is in $W(\mathbb{D})$ and entire function in the variable $\xi,$ and we have
\begin{align}&\|\tau_q^\xi f\|_w\leq (-|\xi|,q)_\infty\|f\|_w .
\end{align}
\end{proposition}
\begin{proof}
Observe that for all $1\leq N< n,$ we can write
\begin{align*}
|(z+q\xi)\dots(z+q^n\xi)|&\leq\prod_{k=1}^n(|z|+|\xi| q^k)\\&
\leq\prod_{k=1}^{N}(|z|+|\xi| q^k)\prod_{k=N+1}^n(|z|+|\xi| q^k)
\\&
\leq(|z|+|\xi| )^N\prod_{k=1}^{n-N}(|z|+|\xi| q^{k+N})
\\&
\leq\big(\frac{|z|+|\xi|}{|z|+|\xi|q^N} \big)^N(|z|+q^N|\xi|)^n,\,\,\,(z,\xi)\neq(0,0).
\end{align*}
Now let $K$ be a compact subset of the unit disk and $R$ a compact subset of the complex plane. There exist a real numbers $0<\rho<1$ and $r>0$ such that for $z\in K$ and $\xi \in R$, we have $|z|<\rho<1$ and $|\xi|<r.$ In addition,  there exists an integer $N$ such that
\begin{equation*}
\rho+q^Nr<1.
\end{equation*}
Then \begin{equation}|(z+q\xi)\dots(z+q^n\xi)|\leq M(\rho+q^Nr)^n,
\end{equation}
where $$M=\max_{z \in K,\,\xi\in R}\big(\frac{|z|+|\xi|}{|z|+|\xi|q^N} \big).$$
This shows the result.
\end{proof}
\begin{lemma}We have
\begin{equation}
\tau_q^\xi f(z) =(-(1-q)\xi D_{q,z};q)_\infty f(z),\label{translation}
\end{equation}
and
\begin{equation}
\tau_q^{\xi}D_{q,z} f(z)=D^+_{q,\xi}\tau_q^{\xi} f(z).\label{lam3}
\end{equation}
\end{lemma}
\begin{proof}From the well-known identity \cite{GR},
\begin{equation}
(a;q)_n=\sum_{k=0}^{n}(-1)^k\left[\begin{array}{c} n
\\k
\end{array}
\right ]_{q}  q^{\binom{k}{2}}a^{k},
\end{equation}
and $\tau_q^\xi z^n$ has the following series expansion
\begin{align*}
\tau_q^\xi z^n&=z^n(-q\xi/z;q)_n
\\&=\sum_{k=0}^{n}\left[\begin{array}{c} n
\\k
\end{array}
\right ]_{q}  q^{\binom{k+1}{2}}\xi^{k}z^{n-k}
\\&=\big(\sum_{k=0}^n\frac{q^{\binom{k+1}{2}}}{[k]_q!}\xi^kD^k_{q,z}q^k\big) z^n\\&= (-(1-q)q\xi D_{q,z};q)_\infty z^n.
\end{align*}
Then \begin{equation}
\tau_q^\xi f(z) =(-(1-q)\xi D_{q,z};q)_\infty f(z).\label{is}
\end{equation}
Now, form \eqref{is} we have

\begin{equation*}
\tau_q^{\xi} D_{q,z}f(z) =\sum_{n=0}^{\infty}\frac{(1-q)^n}{(q;q)_n}q^{\binom{n+1}{2}}\xi^n
D_{q,z}^{n+1}f(z) .
\end{equation*}
Using the relation
$$D_{q}\xi^{n+1}=\frac{1-q^{n+1}}{1-q}\xi^n,$$
we deduce that
\begin{align*}
\tau_q^{\xi} D_{q,z}f(z) &=D_{q,\xi}\big(\sum_{n=0}^{\infty}\frac{(1-q)^{n+1}}{(q;q)_{n+1}}q^{\binom{n+1}{2}}\xi^{n+1}
D_{q,z}^{n+1}f(z)\big)\\&=D_{q,\xi}\tau_q^{\xi/q} f(z)\\&=
D^+_{q,\xi}\tau_q^{\xi} f(z).
\end{align*}
\end{proof}

\subsection{$q$-Duhamel product}
We define the $q$-Duhamel product by
\begin{align}(f\star_q g)(z)=D_q\big(\int_0^z (\tau_q^{-t}f)(z)g(t)\,d_qt\big).\label{Duh}
\end{align}
\begin{lemma}We have\\
\begin{align}(f\star_q g)(z)&=\int_0^z (\tau_q^{-t}D_qf)(z)g(t)\,d_qt+ f(0)g(z)\label{ra1}\\&=\int_0^z (\tau_q^{-t}f)(z)D_qg(t)\,d_qt+ f(z)g(0)\label{ra2}.
\end{align}
\end{lemma}
\begin{proof}Observe that if $F(x,t)$ is a function of two variables, then
\begin{equation}
D_{q,x}\big(\int_0^xF(x,t)\,d_qt\big)=\int_0^xD_{q,x}F(x,t)\,d_qt+F(qx,x)
\end{equation}
Hence,
\begin{equation*}
D_{q,z}\big(\int_0^z (\tau_q^{-t}f)(z)g(t)\,d_qt\big)=\int_0^z (\tau_q^{-t}D_qf)(z)g(t)\,d_qt+ \tau_q^{-z}f(qz)g(x).
\end{equation*}
The result follows from the fact that
$$\tau_q^{-z}f(qz)=f(0).$$
Now to prove \eqref{ra2}, we use formulas \eqref{ra1} and \eqref{lam3} and the q-integration by parts formula and we have
\begin{align*}(f\star_q g)(z)&=-\int_0^zD^+_{q,t} (\tau_q^{-t}f)(z)g(t)\,d_qt+ f(0)g(z)\\&=-f(0)g(z)+f(z)g(0)+\int_0^z (\tau_q^{-t}f)(z)D_qg(t)\,d_qt+ f(z)g(0)\\&=\int_0^z (\tau_q^{-t}f)(z)D_qg(t)\,d_qt+ f(z)g(0).
\end{align*}
\end{proof}
\begin{lemma}We have
\begin{equation}
z^n\star_q z^m=\frac{[n]_q![m]_q!}{[n+m]_q!}\,z^{n+m},\label{con}
\end{equation}
where $$[0]_q=1,\quad [n]_q!=[n]_q\dots[1_q],\,\,n=1,\,2\,\dots\,\,.$$
Clearly \eqref{con} shows
that the $q$-Duhamel convolution is commutative, associative and has $1$ as unit.
\end{lemma}
\begin{proof}From \eqref{Duh}, we can write
\begin{align}
z^n\star_q z^m&=D_q\big(\int_0^z (\tau_q^{-t}z^nt^m\,d_qt\big)
\\&=D_q\big(z^{n+m+1}\int_0^1 (qt;q)_nt^m\,d_qt\big)\\&=
(1-q^{n+m+1})z^{n+m}\sum_{k=0}^\infty q^{k(m+1)}(q^{k+1};q)_n.
\end{align}
Using the formula
\begin{equation}
(q^{k+1};q)_n=\frac{(q;q)_{n+k}}{(q;q)_k}=(q;q)_n\frac{(q^{1+n};q)_k}{(q;q)_k},
\end{equation}
and the $q$-Binomial theorem \cite{GR}
\begin{equation}
\frac{(az;q)_\infty}{(z;q)_\infty}=
\sum_{n=0}^\infty\frac{(a;q)_n}{(q;q)_n}z^n,\,\,|z|<1,
\end{equation}
we get
\begin{align}
z^n\star_q z^m&=z^{n+m}(q;q)_n
(1-q^{n+m+1})\frac{(q^{2+n+m};q)_\infty}{(q^{1+m};q)_\infty}
=\frac{[n]_q![m]_q!}{[n+m]_q!}\,z^{n+m}.
\end{align}
\end{proof}
\begin{theorem}$(W(\mathbb{D}), \star_q)$ is a unital Banach algebra.
\end{theorem}
\begin{proof}Let $$f(z)=\sum_{n=0}^\infty a_nz^n, \quad g(z)=\sum_{n=0}^\infty b_nz^n \in W(\mathbb{D}).$$We have
\begin{align}
\|f\star_qg\|_w&=\sum_{n=0}^\infty|\sum_{k=0}^na_kb_{n-k}\frac{[k]_q![n-k]_q!}{[n]_q!}|
\\&\leq\sum_{n=0}^\infty\sum_{k=0}^n|a_k||b_{n-k}|\frac{[k]_q![n-k]_q!}{[n]_q!}
\end{align}
On the other, from the inequality
\begin{equation}
\frac{1-q^\alpha}{1-q}\leq \alpha q^{(1+\alpha)/2}, \quad \alpha>0,
\end{equation}
we see that
\begin{equation}
\frac{[k]_q![n-k]_q!}{[n]_q!}|\leq\frac{k!(n-k)!}{n!}\leq1.
\end{equation}
Hence,\begin{align}
\|f\star_qg\|_w=\sum_{n=0}^\infty\sum_{k=0}^n|a_k||b_{n-k}|\leq \|f\|_w \|g\|_w
.\end{align}
\end{proof}
\begin{definition}Let  $f(z)=\sum_{n=0}^{\infty}\frac{a_n}{[n]_q!}z^n$ be a holomorphic function. The $q$-Borel transform of $B_qf(z)$
is defined by \begin{equation}(B_qf)(z)=\sum_{n=0}^{\infty}a_nz^n
\end{equation}
\end{definition}
\begin{proposition}
We have \begin{equation}B_q(f\star_q g)=B_q(f)\,B_q( g).
\end{equation}
\end{proposition}
\section{ The $q$-integration operator $V_q$ and its commutant  }
In this section, we will describe the commutant of the $q$-integration operator $V_q$ acting in the Wiener algebra $W(\mathbb{D})$.
Let \begin{equation}
(V_qf)(z):=\int_0^zf(t)\,d_qt=z\star_q f,
\end{equation}
and
\begin{equation}
\mathcal{D}_fg:=f\star_qg.
\end{equation}
\begin{proposition}We have
\begin{equation}
(V_q^nf)(z)=\frac{1}{[n]_q!}z^n\star_q f.
\end{equation}
\end{proposition}
\begin{proof}The proof follows from Al-Salam identity \cite{ism3}
\begin{align*}
(V_q^nf)(z)&=\int_a^z\int_a^{z_n}\dots\int_a^{z_{2}}f(z_1)\,d_qz_1d_qz_2\dots d_qz_n\\&=\frac{(1-q)^{n-1}}{(q;q)_{n-1}}\int_a^zz^n(qt/z;q)_{n-1}f(t)d_qt.
\end{align*}
Hence,
\begin{align*}
(V_q^nf)(z)&
=\frac{(1-q)^{n-1}}{(q;q)_{n-1}}\int_a^z\tau^{-t}z^{n-1}f(t)d_qt\\&=
\frac{1}{[n]_q!}\int_a^z\tau^{-t}D_{q,z}z^{n}f(t)d_qt\\&
=\frac{1}{[n]_q!}z^n\star_q f.
\end{align*}
\end{proof}

\begin{lemma}
The $q$-integration operator $V_q$ is a compact operator in the space $W(\mathbb{D})$.
\end{lemma}
\begin{proof}From the definition of the $q$-Jackson integral \eqref{integ}, we have
\begin{equation}
V_q=z(1-q)\lim _{N\rightarrow \infty}\sum_{n=0}^N q^n \Lambda_q^n
\end{equation}
where $\Lambda_q$ acts on $f\in W(\mathbb{D})$ as follows
$$\Lambda_q f(z)=f(qz).$$
From the fact that
$\Lambda_qz^n=q^nz^n$, for all $n$ ($0<q<1,$)
we see that it is a diagonal operator on $W(\mathbb{D})$ with $q^n\rightarrow 0,\quad n\rightarrow \infty$, then it is compact.\\
On the other hand, it is well-known that a finite composition, finite sum and uniform limit of compact operators is a gain a compact operator and therefore
$$\sum_{n=0}^\infty q^n \Lambda_q^n$$
is a compact operator on $W(\mathbb{D})$.\\ In addition, the multiplication by $z$ is a bounded operator on $W(\mathbb{D})$ . This shows that $V_q$ is a compact operator on $W(\mathbb{D})$.
\end{proof}
\begin{theorem}
The operator $\mathcal{D}_f$ is invertible on $W(\mathbb{D})$ if and only if $f(0)\neq0.$
\end{theorem}
\begin{proof}If $f(z)=\sum_{n=0}^\infty a_nz^n\in W(\mathbb{D})$, we have
\begin{equation}
\mathcal{D}_f=f(0)I+\mathcal{D}_{f-f(0)}
\end{equation}We now prove that the
operator $\mathcal{D}_{f-f(0)}$ is compact. For any fixed integer $N\geq1,$ let us denote $$f_N(z)=\sum_{n=1}^Na_nz^n.$$Then we have
\begin{align*}\mathcal{D}_{f_N}g(z)&=\int_0^z (\tau_q^{-t}g)(z)D_qf_N\,d_qt\\&=
\int_0^z (\tau_q^{-t}g)(z)\sum_{n=1}^N[n]_qa_nt^{n-1}\,d_qt\\&
=\sum_{n=1}^N[n]_qa_n\int_0^z (\tau_q^{-t}g)(z)t^{n-1}\,d_qt\\&=
\sum_{n=1}^N[n]_qa_n\int_0^z (\tau_q^{-t}(z^{n-1})g(t)\,d_qt
\\&=\sum_{n=1}^N[n]_qa_nV_q^ng(z).
\end{align*}
Hence
$$\mathcal{D}_{f_N}=\sum_{n=1}^N[n]_qa_nV_q^n.$$
On the other hand, the operator $V_q$ is compact on $W(\mathbb{D})$ and \begin{align}
\lim_{N\rightarrow \infty}\|\mathcal{D}_{f-f(0)}-\mathcal{D}_{f_N}\|=\lim_{N\rightarrow \infty}\|f-f(0)-f_N\|=0
\end{align}
Hence $\mathcal{D}_{f}$ is a compact operator on $W(\mathbb{D})$, because $\mathcal{D}_{f_N}$ is compact for each $N > 0$.\\We now prove that if $f(0)\neq0$, then $\mathcal{D}_{f}$ is injective. In fact, let $g\in Ker\mathcal{D}_{f}$ , that is $$\mathcal{D}_{f}g(z)=\int_0^z (\tau_q^{-t}D_qf)(z)g(t)\,d_qt+ f(0)g(z)=0.$$Then $\mathcal{D}_{f}g(0)=f(0)g(0)=0$, from which we obtain that $g(0)=0,$
and $$D_{q,z}(\mathcal{D}_{f}g(z))\mid_{z=0}=\big(\int_0^z (D^2_{q,z}\tau_q^{-t}f)(z)g(t)\,d_qt+ f(0)D_{q,z}g(z)+f(0)g(z)\big)\mid_{z=0}$$
Then $$g'(0)=0.$$
Similarly, we show that $g^{(n)}(0)=0,$  which proves that $Ker\mathcal{D}_{f}=\{0\}$. Then by applying Fredholm theorem, we deduce that if $f(0)\neq0 $ then  $\mathcal{D}_{f}$ is  invertible operator on the
space $W(\mathbb{D})$.
\end{proof}
\begin{theorem}We have
\begin{equation}
\{V_q\}'=\{ \mathcal{D}_f, \quad f\in W(\mathbb{D})\}.
\end{equation}
\end{theorem}
\begin{proof}According to the commutativity and associativity
properties of the Duhamel product $\star_q$, we have
\begin{align*}
V_q\mathcal{D}_fg&=z\star_q(f\star_qg)\\&=(z\star_qf)\star_qg\\&
=f\star_q(z\star_q\star_qg)\\&=\mathcal{D}_fV_q.
\end{align*}
Conversely, let $A\in\{V_q\}'$. Then we see that $$AV^n_q=V^n_qA.$$
In particular
$$V^n_qA1=AV^n_q1,\quad \text{for all}\,\,n.$$
Equivalently
$$A(\frac{z^n}{[n]_q!})=A(\frac{z^n}{[n]_q!\star_q1})=AV^n_q1=\frac{z^n}{[n]_q!}\star_q A1.$$
Therefore $$Ap(z)=p(z)\star_q A1$$
for all polynomials $p$. Since by Theorem 6 the algebra $W(\mathbb{D})$ is a Banach algebra with $q$-Duhamel product as multiplication, last equality implies that $Ag=A1*_qg$ for all $g\in W(\mathbb{D})$ (because the set of polynomials is dense in $W(\mathbb{D})$). Clearly, $f=A1\in W(\mathbb{D})$
and hence, $A=\mathcal{D}_f$ , which completes the proof of the theorem.\end{proof}

\end{document}